\documentclass[12pt]{article} \textwidth=16.2cm \textheight=21cm
\usepackage{amsthm}
\usepackage{color}
\usepackage{cite}
\usepackage{graphicx}
\usepackage{psfrag}
\usepackage{url}
\usepackage{stfloats}
\usepackage{amsmath}
\usepackage{amssymb}
\usepackage{graphicx}
\usepackage[titletoc]{appendix}

\newtheorem{theorem}{Theorem}[section]
\newtheorem{definition}{Definition}[section]
\newtheorem{lemma}{Lemma}[section]
\newtheorem{example}{Example}[section]
\newtheorem{remark}{Remark}[section] 
\newtheorem{proposition}{Proposition}[section]

\def\dref#1{(\ref{#1})}

 \def\sin{\mbox{sin}\,}

 \def\dfrac{\displaystyle\frac}

\def\be{\begin{equation}}
\def\bel{\begin{equation}\label}
\def\ee{\end{equation}}
\def\ba{\begin{array}}
\def\ea{\end{array}}
\def\banl{\begin{eqnarray}\label}
\def\ean{\end{eqnarray}}

 \def\bna{\begin{eqnarray}}
\def\ena{\end{eqnarray}} \def\dref#1{(\ref{#1})}

\begin{document}

\title{  Stabilizability Theorems  on Discrete-time Nonlinear  Uncertain Systems  \footnotemark[1]
\author{Zhaobo Liu\footnotemark[2]
, \,Chanying Li \footnotemark[2]
          }
}

\footnotetext[1]{This work was supported in part by the National Natural Science Foundation of China under Grants  11925109  and 11688101.  }

\footnotetext[2]{%
Z.~Liu and C.~Li are with the Key Laboratory of Systems and Control, Academy of Mathematics and Systems Science, Chinese Academy of Sciences, Beijing 100190, P.~R.~China.   They are also with the School of Mathematical Sciences, University of Chinese Academy of Sciences, Beijing 100049, P. R. China.  Corresponding author: Chanying~Li (Email: \texttt{cyli@amss.ac.cn}).}

 \maketitle

\begin{abstract}
This paper derives two  stabilizability theorems  for a basic class  of discrete-time nonlinear  systems with multiple unknown parameters.
First, we claim that a discrete-time  multi-parameter  system is stabilizable if
its nonlinear growth rate is dominated by a polynomial rule. Later,  we  find that a   stabilizable multi-parameter system
 in discrete time is possible to grow exponentially fast. Meanwhile,
optimality and closed-loop identification are also  discussed in this paper.

\end{abstract}

\section{Introduction.}
Adaptive control of
 linear systems (\cite{Astr95}, \cite{cg91}, \cite{goodwin},
\cite{IS96}) and nonlinear systems growing linearly (\cite{TK}, \cite{xieguo00a}) has been a mature topic for  decades, both in continuous time
 and discrete time.
But when it comes to
 systems whose output nonlinearities are faster than linearities, the similarities of adaptive control between continuous- and discrete-time systems  disappear.
 Most  continuous-time nonlinear systems can be globally stabilized by employing
nonlinear damping or back-stepping techniques (\cite{ka} and  \cite{KKK95}), however,
its discrete-time counterpart is not that favored by fortune.
 It was found   in  \cite{guo97} that even for the following basic  discrete-time stochastic system
\begin{equation}\label{sys1}
y_{t+1}=\theta y^b_t+u_t+w_{t+1},\quad \theta \mbox{ is an unknown scalar},
 \end{equation}
 the stabilizability is still possible to be failure.  It showed that the system is stabilizable if and only if $b<4$.
 This fundamental difficulty in discrete-time control was further confirmed by \cite{xieguo1999}, where  system \dref{sys1} is extended to
 the multi-parameter  case:
  \begin{equation}\label{system}
y_{t+1}=\theta_1 y_t^{b_1}+\theta_2 y_t^{b_2}+\cdots+\theta_n
y_t^{b_n}
 +u_t+w_{t+1}.
\end{equation}
Work \cite{xieguo1999}  established an ``impossibility theorem'' by providing a polynomial rule, which  was  proved, a decade later by \cite{lilam13},
  to be a necessary and sufficient condition of the
stabilizability of system \dref{system}. This  polynomial rule also serves as a critical
stabilizability criterion for system  \dref{system}  in the deterministic framework (see \cite{lixieguo06}).
Analogous phenomena arise in the adaptive control of
discrete-time nonparametric nonlinear systems (\cite{Ma08n}, \cite{xieguo00b},
\cite{zhangguo}), semiparametric uncertain systems (\cite{guo2012}, \cite{Sokolov}), linear stochastic systems with
unknown time-varying parameter processes (\cite{xueguohuang}), and continuous-time nonlinear systems with sampled observations (\cite{xueguo}). We refer the readers to \cite{guo1996ls}, 
\cite{Zhao2019Stabilizability},\cite{xu19} for other related works


This paper is intended to extend the results of \cite{lilam13} and \cite{xieguo1999} to the following class of systems:
\begin{equation}\label{system1}
y_{t+1}=\theta_1f_1(y_t)+\theta_2 f_2(y_t)+\cdots+\theta_n
f_n(y_t)
 +u_t+w_{t+1}.
\end{equation}
We conjecture that system \dref{system1} is stabilizable if the nonlinear growths of $f_1,\ldots,f_n$ are dominated by
 some power functions $x^b_1,\ldots,x^{b_n}$ respectively, where $b_1,\ldots,b_n$ satisfy the polynomial rule referred.
  Comparing system \dref{system} and  system \dref{system1}, a significant difference is that $f_1,\ldots,f_n$  in \dref{system1}
  may be very close or intersect infinitely often, while functions $x^{b_1},\ldots,x^{b_n}$ with $b_1>\cdots>b_n>0$ in system \dref{system} are away from each other
   when $x$ is large. Intuitively,
 this will cause some obstacles in  the closed-loop identification for  system \dref{system1}.  And then, the stabilizability might be affected.
 By establishing an inequality on the minimal eigenvalue of the inverse information  matrix in Proposition \ref{tzz1},
we prove our conjecture in Theorem \ref{1sta}.

 Theorem \ref{1sta} requires that  system \dref{system1}  grows no faster than some power function.
But this is not the growth rate limit for the stabilizability of system \dref{system1}. For the scalar case ($n=1$),
recall that
\cite{liu2018} asserts $f_1(x) = O(|x|^{b_1})$
with $b_1<4$ is only required for a very tiny fraction of $x$ in $\mathbb{R}$, even if it grows exponentially
fast for the other $x$. Is it  true for the multi-parameter  case?
We prove in this paper that a multi-parameter  stabilizable system still has a chance to grow exponentially. With the help of
 the proposed inequality in Proposition \ref{tzz1}, we again find that the stabilizability of system \dref{system1} can be achieved if $f_1,\ldots,f_n$
 are bounded on a  tiny fraction of $x$ in $\mathbb{R}$,  while these functions may grow exponentially
fast for the other $x$.

The  paper is built up as follows.  Section \ref{MR} presents two stabilizability theorems and    Section \ref{CLI} discusses the
corresponding  closed-loop identification. The proofs of the main results are included in Sections \ref{soc}--\ref{AppA1}.
\section{Global Stabilizability}\label{MR}
We study the following discrete-time nonlinear  system with multiple unknown  parameters:
\begin{eqnarray}\label{sys21}
 y_{t+1}=\theta ^{\tau}\phi(y_t)+u_t+w_{t+1},~~~~~t\geq0,
\end{eqnarray}
where $\theta=(\theta_1,\ldots,\theta_n)^{\tau}\in \mathbb{R}^{n},  n\geq 2$ are unknown parameters, $y_t, u_t, w_t$ are the  output, input and noise signals, respectively. Assume
that $\phi=(f_{1},\ldots,f_{n})^\tau : \mathbb{R}\to \mathbb{R}^{n}$ is a known measurable vector function, where $f_{j} \in C^{n}(E)$,   $1\leq j\leq n$,  and  $E$ is an open set in $\mathbb{R}$. We rewrite \dref{sys21} as
\begin{eqnarray}\label{sys2}
 y_{t+1}=\sum_{j=1}^{n}\theta_jf_{j}(y_t)+u_t+w_{t+1},~~~~~t\geq0,
\end{eqnarray}
and present the definition of stabilizability in the following sense.
\begin{definition}
System \dref{sys2} is said to be almost surely globally stabilizable,  if there exits a feedback control law
\begin{equation}\label{utlaw}
u_t\in \mathcal{F}_{t}^y\triangleq\sigma\lbrace y_i,0\leqslant i\leqslant t\rbrace,~t=0,1,\ldots
 \end{equation}such that for any initial conditions $y_0\in \mathbb{R},$
 \begin{equation}\label{wd}
 \sup_{t\geq 1}\frac{1}{t}\sum_{i=1}^{t}y_i^2<+\infty,\quad \mbox{a.s.}.
 \end{equation}
\end{definition}
 We analyze our problem  
 in some standard assumptions below.

\begin{description}

\item[A1]
The noise $\{w_t\} $ is an  i.i.d random sequence with  $w_1 \sim N(0,\sigma^2)$.

\item[A2] Parameter $ \theta\sim N(\theta_0,I_{n})$ is independent of $\{w_t\}$.

\item[A3]  $f_{1},\ldots,f_{n}$ are linearly independent on $E$.
\end{description}
\begin{remark}
We consider a typical case where $E=\mathbb{R}$. If $f_{j}\equiv 0$ for all $j\in [1,n]$, system \dref{sys2} degenerates to $y_{t+1}=w_{t+1}$. Otherwise,
with no loss of generality, let $f_{1},\ldots,f_{k}$ be linearly independent on $\mathbb{R}$, $1\leq k\leq n$, such that every  $f_l, l\in [k+1,n]$ is a linear combination of  $f_{1},\ldots,f_{k}$. Consequently, there are $n-k$ unit vectors $(x_{1,l},\ldots,x_{k,l})^{\tau}$ satisfying $f_{l}(y)=\sum_{j=1}^{k}x_{j,l}f_{j}(y)$, $l\in[k+1,n]$. Therefore, by letting
\begin{equation*}
\theta_j'\triangleq{\theta_j+\sum_{l=k}^{n}x_{j,l}\theta_l},\quad 1\leq j\leq k,
\end{equation*}
system \dref{sys2} becomes
\begin{equation}\label{jj1}
 y_{t+1}=\sum_{j=1}^{k}\theta_j'f_{j}(y_t)+u_t+w_{t+1},~~~~~t\geq0.
 \end{equation}
This means it suffices to discuss system \dref{jj1}.  So, Assumption A3 is a natural condition.
\end{remark}
Our first theorem below  extends the result of \cite{lilam13} to a more general situation. 
\begin{theorem}\label{1sta}
Under Assumptions A1--A3,  system \dref{sys2} is globally stabilizable if
\begin{equation*}
f_{j}(x)=O(|x|^{b_j})+O(1),\quad  1\leq j\leq n,
\end{equation*}
where $b_1>b_2>\cdots>b_{n}>0$ are $n$ numbers satisfying $b_1>1$ and
\begin{equation}\label{px}
P(x)=x^{n+1}-b_1x^n+(b_1-b_2)x^{n-1}+\cdots+b_{n}>0,\quad x\in(1,b_1).
\end{equation}
\end{theorem}

\begin{example}
Under Assumptions A1--A2, consider system \dref{sys2}  with
\begin{eqnarray*}
f_{1}(x)=x^2\cos x\quad\mbox{and}\quad f_{2}(x)=x\sin x.
\end{eqnarray*}
The images of $f_{1}$ and $f_{2}$  intersect each other infinitely many times. The stabilizability issue of such systems cannot be covered by the existing theory.
Now,  applying Theorem \ref{1sta} with $b_1=2$ and $b_2=1$, we immediately conclude that the system is stablizable.  
\end{example}

For the sake of stabilizability,  Theorem \ref{1sta} requires that  system \dref{sys2}  grows no faster than some power function. On the other hand, for the scalar-parameter case,  \cite{liu2018} finds the corresponding system 
is possible to be stabilized when growing exponentially. But the number of the unknown parameters  affects the allowed growth rate of a stabilizable system  (see \cite{lilam13}). So,  we wonder wether a multi-parameter   stabilizable system still has a chance to grow exponentially? The following theorem gives an affirmative answer.

\begin{theorem}\label{2sta}
Under Assumptions A1--A3,  system \dref{sys21} is globally stabilizable if \\
 (i) for some $k_1,k_2>0$,
\begin{equation}\label{fk1k2}
\|\phi(x)\|\leqslant k_1 e^{k_2|x|},\quad \forall x\in \mathbb{R};
\end{equation}
(ii) there exists a number $L>0$ such that for $S_L\triangleq{\lbrace x: \|\phi(x)\|\leq L\rbrace}$,
\begin{equation}\label{sl}
\liminf_{l\rightarrow+\infty}\frac{\ell(S_L\cap[-l,l])}{l}>0,
\end{equation}
where $\ell$ denotes the Lebesgue measure.
\end{theorem}

Clearly, $p_{L}\triangleq\liminf_{l\rightarrow+\infty}\frac{\ell(S_L\cap[-l,l])}{l}$  describes the proportion of $S_L$ in $\mathbb{R}$.  Since $p_L>0$ can be taken as small as one likes  in Theorem \ref{2sta},  a stabilizable  system may possess a very sparse $S_L$. We give an extreme example to illustrate it.


\begin{example}
Under Assumptions A1--A2, consider system \dref{sys2} with
\begin{eqnarray*}
f_{1}(x)=1+e^x\cdot I_{\lbrace sin x >-0.999\rbrace}\quad\mbox{and}\quad f_{2}(x)=e^{2x}\cdot I_{\lbrace sin x >-0.999\rbrace}.\nonumber
\end{eqnarray*}
Clearly, \dref{sl} holds for $L=1$. The system  is thus stabilizable by virtue of Theorem \ref{2sta}. We remark  that this system grows exponentially fast on  most part of the real line.
\end{example}

\section{Closed-loop Identification}\label{CLI}

In order to achieve the stabilization of system \dref{sys2}, we employ the self-tuning regulator (STR) based on the least-squares  (LS) algorithm. The standard LS estimate $\hat{\theta}_{t}$ for parameter $\theta$ can be recursively defined by
\begin{eqnarray}\label{LS2}
\left\{
\begin{array}{l}
\hat{\theta}_{t+1}=\hat{\theta}_{t}+\sigma^{-2}P_{t+1}\phi_t(y_{t+1}-u_t-\phi_t^\tau\hat{\theta}_{t})\\
P_{t+1}=P_t-(\sigma^2+\phi_t^\tau P_t\phi_t)^{-1}P_t\phi_t\phi_t^\tau P_t,~~P_0=I_{n}\\
\phi_t\triangleq{\phi(y_t)},~~t\geq0
\end{array},
\right.
\end{eqnarray}
where initial vectors $\hat{\theta}_{0}=\theta_0$ and $\phi_0$ are taken   random. In light of the ``certainty equivalence principle'', the controller is designed as follows:
\begin{equation} \label{ut}
u_t=-\hat{\theta}_{t}^\tau\phi_t,\quad t\geq0.
\end{equation}
We shall show in the next two sections that the LS-STR \dref{LS2}--\dref{ut} is the desired stabilizing controller for both Theorems \ref{1sta} and \ref{2sta}. Besides, during the control process, parameter $\theta$ can be identified simultaneously.
\begin{theorem}\label{1iden}
Under the conditions of Theorem \ref{1sta}, the LS estimator is strong consistent in the closed-loop  system \dref{sys2}, \dref{LS2} and \dref{ut}. More precisely,
$$
\|\hat{\theta}_{t+1}-\theta\|^2=O\left(\frac{\log t}{t}\right),\quad \mbox{a.s.}.
$$
\end{theorem}

\begin{theorem}\label{2iden}
Under the conditions of Theorem \ref{2sta}, the LS estimator is strong consistent in the closed-loop  system \dref{sys21}, \dref{LS2} and \dref{ut}. More precisely,
$$
\|\hat{\theta}_{t}-\theta\|^2=O(t^{-\frac{1}{2}}),\quad \mbox{a.s.}.
$$
\end{theorem}

It is worth mentioning that  for our situation,    the strong consistency of the LS estimates in the closed-loop system 
  can be guaranteed by the stability of the system. We now discuss it in details.

Let $\lambda_{\min}(t+1)$ be the minimal eigenvalue of $P_{t+1}^{-1}$ defined in \dref{LS2}. 
Under Assumptions A1--A2, \cite{ef63} and \cite{St77} imply
\begin{eqnarray}\label{77}
\left\lbrace \lim_{t\rightarrow+\infty}\lambda_{\min}(t+1)=+\infty\right\rbrace=\left\lbrace\lim_{t\rightarrow+\infty}\hat{\theta}_{t}=\theta\right\rbrace.
\end{eqnarray}
Furthermore, 
\cite{laiwei82} and \cite
{guo95} point out  that
\begin{equation}\label{jd}
\|\hat{\theta}_{t+1}-\theta\|^2=O\left(\frac{\log\left(1+\sum_{i=0}^{t}\|\phi(y_i)\|^2\right)}{\lambda_{\min}(t+1)}\right),\quad \mbox{a.s.},
\end{equation}
which  provides a powerful tool to estimate the convergence rate of $\lbrace \hat{\theta}_{t}\rbrace_{t\geq 0}$. 
By the help of \dref{jd}, we can derive the following lemma with the proof  stated in Appendix \ref{AppA}.
\begin{lemma}\label{c1}
Under Assumptions A1--A2, let the closed-loop system \dref{sys21}, \dref{LS2}, \dref{ut}  satisfy
\begin{equation}\label{bx}
\|\phi(x)\|=O(|x|^b)+O(1),\quad b>0.
\end{equation}
Then, there is a set $D$ with $P(D)=0$ such that
\begin{gather}\label{opti}
\left\lbrace\sum_{i=1}^{t}y_{i}^2=O(t)\right\rbrace \backslash D\subseteq  \left\lbrace\liminf_{t\rightarrow+\infty}\frac{\lambda_{\min}(t+1)}{t}>0
\right\rbrace\cap\left\lbrace\lim_{t\rightarrow+\infty}\frac{1}{t}\sum_{i=1}^{t}y_i^2=\sigma^2\right\rbrace,
\end{gather}
\begin{gather}\label{theconv}
\sum\nolimits_{i=1}^{t}y_i^2=O(t)\quad \mbox{implies}\quad \|\hat{\theta}_{t+1}-\theta\|^2=O\left(\frac{\log t}{t}\right),\quad \mbox{a.s.}.
\end{gather}

\end{lemma}

\begin{remark}
By virtue of \dref{opti} in Lemma \ref{c1}, 
$$
\sum_{i=1}^{t}y_{i}^2=O(t) \quad \mbox{is equaivalent to}\quad   \lim_{t\rightarrow+\infty}\frac{1}{t}\sum_{i=1}^{t}y_i^2=\sigma^2 \quad \mbox{a.s.},
$$
which is referred to as optimality (see \cite{guo95}). Meanwhile, \dref{theconv} suggests that  Theorem     \ref{1iden} is a direct result of  Theorem     \ref{1sta}.
\end{remark}

For systems growing exponentially, we remark that the proof of Theorem \ref{2sta} indicates
\begin{eqnarray}\label{inflam}
\liminf_{t\rightarrow+\infty}\frac{\lambda_{\min}(t+1)}{t}>0,\quad \mbox{a.s.}.
\end{eqnarray}
Hence,  \dref{jd} infers
\begin{eqnarray*}
\|\hat{\theta}_{t+1}-\theta\|^2=O\left(\frac{\log\left(1+\sum_{i=0}^{t}\|\phi(y_i)\|^2\right)}{t}\right)
=O\left(\frac{\sqrt{\sum_{i=1}^{t}y_{i}^2}}{t}\right),\quad \mbox{a.s.}.\nonumber
\end{eqnarray*}
Consequently, it is straightforward that
\begin{lemma}\label{c2}
If the conditions of Theorem \ref{2sta} hold, then in the closed-loop system \dref{sys21}, \dref{LS2}, \dref{ut},
$$ \sum\nolimits_{i=1}^{t}y_i^2=O(t)\quad \mbox{implies}\quad \|\hat{\theta}_{t}-\theta\|^2=O(t^{-\frac{1}{2}}),\quad \mbox{a.s.}.$$
\end{lemma}

The remainder of the proof is thus   focused on the stability of the closed-loop system  \dref{sys21}, \dref{LS2} and \dref{ut}.
We shall see that \dref{inflam} plays a core role not only in the above closed-loop identification, but also in justifying Theorems \ref{1sta}--\ref{2sta}. We close this section by presenting an important proposition, whose proof is included in Appendix \ref{AppA}.
\begin{proposition}\label{tzz1}
Under Assumptions A1--A3, there is a constant $M>0$ such that in the closed-loop system \dref{sys21}, \dref{LS2}, \dref{ut}, 
\begin{eqnarray*}
\liminf_{t\rightarrow+\infty}\frac{\lambda_{\min}(t+1)}{t}\geq M \liminf_{t\rightarrow+\infty}\frac{1}{t}\sum_{i=1}^{t}\frac{1}{\sigma_{i-1}}\quad \mbox{a.s.}.
\end{eqnarray*}
\end{proposition}

\section{Proof of Theorem \ref{1sta}}\label{soc}

For the closed-loop system \dref{sys2}, \dref{LS2} and \dref{ut}, one has
\begin{eqnarray}\label{yt}
&&P_{t+1}^{-1}=I_{n}+\frac{1}{\sigma^2}\sum_{i=0}^{t}\phi_i\phi
_i^{\tau},\nonumber\\
&&y_{t+1}=\tilde{\theta_t}f(y_t)+w_{t+1},
\end{eqnarray}
where $\tilde{\theta_t}\triangleq\theta-\hat{\theta}_{t}$, $t\geq0$. Since the LS algorithm \dref{LS2} is exactly  the standard Kalman filter in our case, it yields that $\hat{\theta}_{t}=E[\theta|\mathcal{F}_t^y]$ and $P_t=E[\tilde{\theta_t}^\tau\tilde{\theta_t}|\mathcal{F}_t^y]$. Hence,  for each $t\geq0$,  $y_{t+1}$ possesses a  conditional    Gaussian distribution given $\mathcal{F}_t^y$. The conditional mean and variance are
\begin{eqnarray}\label{mean}
&&m_t\triangleq{E[y_{t+1}|\mathcal{F}_t^y]}=u_t+\hat{\theta}_{t}^\tau\phi_t=0\\
&&\label{var}
\sigma_t^2\triangleq {Var(y_{t+1}|\mathcal{F}_t^y)}=\sigma^2+\phi_t^\tau P_t\phi_t=\sigma^2\cdot\frac{|P_{t+1}^{-1}|}{|P_{t}^{-1}|},\quad\mbox{a.s.}.
\end{eqnarray}
So, we shall make  efforts to prove $\sup_{t}\sigma_t<+\infty$ almost surely. To this end,
 we provide several relevant lemmas. 

\begin{lemma}\label{one}
Under the conditions of Theorem \ref{1sta},  assume that events $\lbrace |P_{t}^{-1}|< (1+\sigma^{-2})^{t}\rbrace_{t\geq 1}$ occur finitely on some set $D$ with probability $P(D)>0$,
 then
\begin{equation}\label{yxd}
\sup_{t}\sigma_t<+\infty,\quad\mbox{a.s.}\quad\mbox{on}~D.
\end{equation}
\end{lemma}
\begin{proof}
At first, we define some random matrices:
\begin{eqnarray*}
\left\{
\begin{array}{l}
Q_{0}^{-1}=I_n\\
Q_{k}^{-1}=Q_{k-1}^{-1}+\frac{1}{\sigma^2}\cdot\phi_{t_{k-1}}\phi_{t_{k-1}}^{\tau},\quad k\geq 1
\end{array},
\right.
\end{eqnarray*}
where the random subscript $t_{k}$ with $t_{0}=0$ satisfies
\begin{eqnarray*}
\left\{
\begin{array}{l}
\phi_{t_{k}}^{\tau}Q_{k}\phi_{t_{k}}>\phi_{t_{k-1}}^{\tau}Q_{k-1}\phi_{t_{k-1}}\\
\phi_{t}^{\tau}Q_{k}\phi_{t}\leq\phi_{t_{k-1}}^{\tau}Q_{k-1}\phi_{t_{k-1}},\quad ~~  t_{k-1}<t<t_{k}
\end{array},\quad k\geq 1.
\right.
\end{eqnarray*}
If $\lbrace t_k\rbrace$ is a finite sequence, then there is a $k$ such that
\begin{eqnarray*}
\phi_{t}^{\tau}Q_{k+1}\phi_{t}\leq\phi_{t_{k}}^{\tau}Q_{k}\phi_{t_{k}},\quad \forall ~t>t_{k}.
\end{eqnarray*}
Consequently,
\begin{eqnarray*}
\sigma_{t}^2&=&\sigma^2\cdot\frac{|P_{t+1}^{-1}|}{|P_{t}^{-1}|}=\sigma^2+\phi_{t}^{\tau}P_{t}\phi_{t}\nonumber\\
&\leq &\sigma^2+\phi_{t}^{\tau}Q_{k+1}\phi_{t}\leq \sigma^2+\phi_{t_{k}}^{\tau}Q_{k}\phi_{t_{k}},\quad \forall ~t>t_{k},\nonumber
\end{eqnarray*}
which leads to $\sup_{t}\sigma_t<+\infty$.

Now, we assume that there exists a set $D'\subset D$ with $P(D')>0$ such that $\lbrace t_k\rbrace$ is  infinite on $D'$. Clearly,
\begin{eqnarray*}
\frac{|Q_{k}^{-1}|}{|Q_{k-1}^{-1}|}&=& \sigma^2+\phi_{t_{k-1}}^{\tau}Q_{k-1}\phi_{t_{k-1}}\nonumber\\
&<&\sigma^2+\phi_{t_{k}}^{\tau}Q_{k}\phi_{t_{k}}=\frac{|Q_{k+1}^{-1}|}{|Q_{k}^{-1}|},\quad k\geq 1.
\end{eqnarray*}
Similarly to \cite[Lemma 3.1]{lilam13}, we can prove that for any $t\in(t_{k-1},t_{k}]$,
\begin{eqnarray*}
\frac{|P_{t}^{-1}|}{|P_{t-1}^{-1}|}\leq \frac{|Q_{k}^{-1}|}{|Q_{k-1}^{-1}|}.
\end{eqnarray*}
On the other hand, since
\begin{eqnarray}
\sum_{t=1}^{+\infty}P(|y_t|> \sigma_{t-1}\log t |\mathcal{F}_{t-1}^{y})=\frac{1}{\sqrt{2\pi}}\sum_{t=1}^{+\infty}\int_{|x|\geq \log t}e^{-\frac{x^2}{2}}\,dx<+\infty,\nonumber
\end{eqnarray}
by \textit{Borel-Cantelli-Levy} theorem, the events $\lbrace |y_t|> \sigma_{t-1}\log t\rbrace$ occur only finite many times for  $t\geq 1$. That is,  for all  sufficiently large $t$,
\begin{equation}\label{ylog}
|y_t|^2\leq \sigma_{t-1}^2\log^2 t=\sigma^2\cdot\frac{|P_{t}^{-1}|}{|P_{t-1}^{-1}|}\log^2 t,\quad\mbox{a.s.},
\end{equation}
which infers that there exists a random number $\gamma>1$ such that for all  $t\geq 0$,
$$|y_t|^2\leq \gamma\cdot\sigma_{t-1}^2\log^2 (t+3),\quad\mbox{a.s.}.$$

Next, for $d\geq1$, we define
\begin{equation}
\alpha_d(j)\triangleq {\frac{1}{\sigma^2}\cdot(f_{1}(y_{t_{j}})f_{d}(y_{t_{j}}),\ldots,f_{n}(y_{t_{j}})f_{d}(y_{t_{j}}))^{\tau}},\quad j\geq 0,\nonumber
\end{equation}
and $\alpha_d(-1)\triangleq {e_d}$,  where $ e_d$ is the $d\mbox{th}$ column of the identity matrix $I_{n}$.
Then,
\begin{eqnarray}\label{detr1}
|Q_{k+1}^{-1}|&=&\det\left(\sum_{j=-1}^{k}\alpha_1(j),\ldots,\sum_{j=-1}^{k}\alpha_n(k)\right)\nonumber\\
&=& \sum_{s_1,\ldots,s_{n}=-1}^{k}\det\left(\alpha_1(s_1),\ldots,\alpha_n(s_{n})\right).
\end{eqnarray}
If there exist two subscripts $d\neq d'$ such that $s_d=s_{d'}\neq -1$, we obtain
\begin{equation*}
\det\left(\alpha_1(s_1),\ldots,\alpha_{n}(s_{n})\right)=0,
\end{equation*}
and hence
\begin{eqnarray}\label{detr}
|Q_{k+1}^{-1}|&=& \sum_{(s_1,\ldots,s_{n})\in \mathcal{W}(k)}\det\left(\alpha_1(s_1),\ldots,\alpha_{n}(s_{n})\right).
\end{eqnarray}
Here,  given integer $k\geq 1$ and positive integers $l_1<\cdots<l_m$,
\begin{eqnarray}
\mathcal{W}(k)&\triangleq &\lbrace(l_1,\ldots,l_{n}):l_i\in\lbrace-1,0,\ldots,k\rbrace,i\in[1,n];\nonumber\\
&& ~~~~~~~~~~~~~~~~~~~~~~~~~~~~~~ l_i\neq l_{i'} ~\mbox{if} ~i\neq i',~ l_{i'}\not=-1\rbrace,\nonumber
\end{eqnarray}
\begin{eqnarray}
\mathcal{H}_k^{(l_1,\ldots,l_m)}&\triangleq &{\lbrace(i_1,\ldots,i_k): i_j\in \lbrace l_1,\ldots,l_m\rbrace,1\leq j\leq k;i_r\neq i_s~\mbox{if}~r\neq s\rbrace}.\nonumber
\end{eqnarray}
Now, for any $(s_1,\ldots,s_{n})\in \mathcal{W}(k)$,
\begin{eqnarray}\label{fz}
&&\det\left(\alpha_1(s_1),\ldots,\alpha_n(s_{n})\right)\nonumber\\
&\leq & \sum_{(l_1,\ldots,l_{n})\in \mathcal{H}_{n}^{(1,\ldots,n)}}\prod_{k\in[1,n],s_k\neq -1}\frac{1}{\sigma^2}\cdot|f_{l_k}(y_{t_{s_{k}}})f_{k}(y_{t_{s_k}})|\nonumber\\
&\leq & \sum_{(l_1,\ldots,l_{n})\in \mathcal{H}_{n}^{(1,\ldots,n)}} \prod_{k\in[1,n],s_k\neq -1}\frac{1}{\sigma^2}\cdot(L_1+L_2|y_{s_{k}}|^{b_{l_k}})\cdot(L_1+L_2|y_{s_k}|^{b_k})\nonumber\\
&\leq & \left(L_1+L_2\right)^{2n}\cdot\sum_{(l_1,\ldots,l_{n})\in \mathcal{H}_{n}^{(1,\ldots,n)}} \prod_{k\in[1,n],s_k\neq -1}\nonumber\\
&&~~~~~~~~~~~~~~~~~~~~~~~~~~~~~~~~~~~~\frac{1}{\sigma^2}\cdot\left(\gamma\cdot\log^2 (t_{s_k}+3)\cdot\sigma^2\cdot\frac{|P_{t_{s_k}}^{-1}|}{|P_{t_{s_k}-1}^{-1}|}\right)^{\frac{b_{l_k}+b_k}{2}}\nonumber\\
&\leq & (L_1+L_2)^{2n}\cdot\sum_{(l_1,\ldots,l_{n})\in \mathcal{H}_{n}^{(1,\ldots,n)}}\prod_{k\in[1,n],s_k\neq -1}\nonumber\\
&&~~~~~~~~~~~~~~~~~~~~~~~~~~~~~~~~~~~~\frac{1}{\sigma^2}\cdot\left(\gamma\cdot\log^2 (t_{s_k}+3)\cdot\sigma^2\cdot\frac{|Q_{s_k}^{-1}|}{|Q_{s_k-1}^{-1}|}\right)^{\frac{b_{l_k}+b_k}{2}}\nonumber\\
&\leq &(L_1+L_2)^{2n}\cdot(1+\sigma^{2b_1-2}+\sigma^{2b_{n}-2})^{n}\cdot n!\nonumber\\
&&~~~~~~~~~~~~~~~~~~~~~~~~~~~\cdot(\gamma\cdot\log^2 (t_{s_k}+3))^{(b_1+\cdots+b_{n})}\cdot\prod_{i=1}^{n}\left(\frac{|Q_{k+1-i}^{-1}|}{|Q_{k-i}^{-1}|}\right)^{b_i},\nonumber\\
\end{eqnarray}
where $Q_{-1}\triangleq{I_{n}}$, and $L_1,L_2$ are two positive numbers satisfying
\begin{equation*}
|f_{j}(x)|\leq L_1+L_2|x|^{b_j},\quad \forall x\in\mathbb{R},~ j\in[1,n].
\end{equation*}
By combining  \dref{detr} and \dref{fz}, we conclude
\begin{eqnarray}
|Q_{k+1}^{-1}|&\leq &(k+2)^{n}\cdot(L_1+L_2)^{2n}\cdot(1+\sigma^{2b_1-2}+\sigma^{2b_{n}-2})^{n}\nonumber\\
&&\cdot n!\cdot(\gamma\cdot\log^2 (t_{k}+3))^{\sum_{i=1}^{n}b_i}\cdot\prod_{i=1}^{n}\left(\frac{|Q_{k+1-i}^{-1}|}{|Q_{k-i}^{-1}|}\right)^{b_i}.\nonumber
\end{eqnarray}
As a consequence, if $|Q_{k+1}^{-1}|>t_k^{\sqrt{\log t_k}}$ for all sufficiently large $k$, then there must exist a random number $t_{\epsilon}'$  for any given $\epsilon>0$ such that
\begin{eqnarray}\label{1jq}
&&(1-\epsilon)\log |Q_{k+1}^{-1}|\nonumber\\
&\leq &\sum_{i=1}^{n}b_i(\log |Q_{k+1-i}^{-1}|-\log |Q_{k-i}^{-1}| )\nonumber\\
&=& b_1 \log |Q_{k}^{-1}|-\sum_{i=1}^{n-1}(b_i-b_{i+1})\log |Q_{k-i}^{-1}|-b_{n}\log |Q_{k-n}^{-1}|
,\quad k\geq t_{\epsilon}'.\nonumber\\
\end{eqnarray}
Define $z_k\triangleq{\log |Q_{k+1}^{-1}|/ \log |Q_{k}^{-1}|}$, $k\geq 1$ and $z\triangleq{\liminf_{k\rightarrow+\infty}z_k}\geq 1$. Then inequality \dref{1jq} is equivalent to
\begin{equation*}
1-\epsilon+\sum_{i=1}^{n-1}(b_i-b_{i+1})\frac{1}{\prod_{j=0}^{i}z_{k-j}}+b_{n} \frac{1}{\prod_{j=0}^{n}z_{k-j}}\leq b_1\frac{1}{z_k},\quad k\geq t_{\epsilon}'.
\end{equation*}
Taking limit superior on both sides of the above inequality, we have
\begin{equation*}
1-\epsilon+\sum_{i=1}^{n-1}(b_i-b_{i+1})\frac{1}{z^{i+1}}+b_{n}\frac{1}{z^{n}}\leq b_1\frac{1}{z}.
\end{equation*}
Letting $\epsilon\rightarrow+\infty$ shows that $P(z)\leq 0$ and $z>1$. This contradicts to the definition of $P(x)$. Hence, we immediately deduce that
\begin{equation}\label{qkd}
|Q_{k+1}^{-1}|\leq t_k^{\sqrt{\log t_k}}\quad \mbox{i.o.}\quad\mbox{a.s.~on~} D'.
\end{equation}
Similarly to \dref{detr1}--\dref{fz}, when $k$ is sufficiently large and satisfies $|Q_{k+1}^{-1}|\leq t_k^{\sqrt{\log t_k}}$, for any $t\in (t_k+1,t_{k+1}+1]$, we have
\begin{eqnarray}
|P_{t}^{-1}|&\leq &(L_1+L_2)^{2n}\sum_{(s_1,\ldots,s_{n})\in \mathcal{W}(t-1)} \sum_{(l_1,\ldots,l_{n})\in \mathcal{H}_{n}^{(1,\ldots,n)}}\nonumber\\
&&\prod_{k\in[1,n],s_k\neq -1}\frac{1}{\sigma^2}\left(\gamma\cdot\log^2 (s_k+3)\cdot\sigma^2\cdot\frac{|P_{s_k}^{-1}|}{|P_{s_k-1}^{-1}|}\right)^{\frac{b_{l_k}+b_k}{2}}\nonumber\\
&\leq &(L_1+L_2)^{2n}\cdot(1+\sigma^{2b_1-2}+\sigma^{2b_{n}-2})^n\nonumber\\
&&\cdot(\gamma\cdot \log^2 (t+2))^{\sum_{i=1}^{n}b_i}(t+1)^{n} \cdot n!\cdot \left(\frac{|Q_{k+1}^{-1}|}{|Q_{k}^{-1}|}\right)^{\sum_{i=1}^{n}b_i}\nonumber\\
&\leq &(L_1+L_2)^{2n}\cdot(1+\sigma^{2b_1-2}+\sigma^{2b_{n}-2})^{n}\nonumber\\
&&\cdot(\gamma\cdot \log^2 (t+2))^{\sum_{i=1}^{m_1}b_i}(t+1)^{n} \cdot n!\cdot t_k^{\sqrt{\log t_k}\cdot\sum_{i=1}^{n}b_i}\nonumber\\
&<& (1+\sigma^{-2})^{ t}.\nonumber
\end{eqnarray}
This together with \dref{qkd} leads to
\begin{equation*}
|P_{t}^{-1}|<  (1+\sigma^{-2})^{ t}\quad \mbox{i.o.}\quad\mbox{a.s.~on~}D'.
\end{equation*}
However, according to the assumption, events $\lbrace |P_{t}^{-1}|< (1+\sigma^{-2})^{t}\rbrace_{t\geq 1}$ occur  finitely on $D$, which arises a contradiction. Hence $\lbrace t_k\rbrace$ is finite on $D$ almost surely, and \dref{yxd} follows.
\end{proof}

\begin{lemma}\label{two}
Under Assumptions A1--A3,  assume  \dref{bx} holds and  there is a set $D$ with $P(D)>0$ such that
\begin{equation}\label{P<io}
|P_{t}^{-1}|<(1+\sigma^{-2})^{t}\quad \mbox{i.o.}\quad\mbox{a.s.~on~}D.
\end{equation}
Then,
\begin{equation}\label{jj}
\sup_{t}\sigma_t<+\infty,\quad
\mbox{a.s.}\quad\mbox{on}~D.
\end{equation}
\end{lemma}
\begin{proof}
Denote
\begin{eqnarray*}
F\triangleq{\lbrace t\geq 0: |P_{t}^{-1}|\geq (1+\sigma^{-2})^{t},|P_{t+1}^{-1}|< (1+\sigma^{-2})^{t+1}\rbrace}.\nonumber
\end{eqnarray*}
Firstly, by \dref{P<io}, for all sufficiently large $t$
$$
|P_{t}^{-1}|<(1+\sigma^{-2})^{t},\quad \mbox{a.s.}\quad\mbox{on}~\lbrace |F|<+\infty \rbrace\cap D.
$$
Let $\varepsilon=\frac{M}{3}\cdot\min\lbrace (1+\sigma)^{-1},\sigma^3\cdot(1+\sigma^2)^{-1}\rbrace$.
In view of Lemma \ref{tzz1}, there is a random integer $t_1$ such that for all $t>t_1$,
\begin{eqnarray*}
\lambda_{\min}(t+1)&\geq & M\sum_{i=1}^{t}\frac{1}{\sigma_{i-1}}-\varepsilon t\geq M\cdot t\cdot\left(\frac{1}{\prod_{i=1}^{t}\sigma_{i-1}}\right)^{\frac{1}{t}}-\varepsilon t\nonumber\\
&=&M\sigma\cdot t\cdot\left(\frac{1}{|P_{t}^{-1}|}\right)^{\frac{1}{2t}}-\varepsilon t>M\sigma\cdot t\cdot\left(\frac{1}{(1+\sigma^{-2})^{t}}\right)^{\frac{1}{2t}}-\varepsilon t\nonumber\\
&>&\varepsilon t,\quad \mbox{a.s.}\quad\mbox{on}~\lbrace |F|<+\infty \rbrace\cap D.
\end{eqnarray*}
In addition, for some integer $N>0$,
\begin{equation}\label{tn}
\varepsilon t>n(K_1+K_2)+nK_{2}\log^{2b} t \cdot (1+\sigma)^{2b},\quad \forall t\geq N,
\end{equation}
where constants $K_1,K_2$ satisfy $\|\phi(x)\|\leq K_1+K_2|x|^b$ for $x\in\mathbb{R}$. Clearly, there exists a random integer $t_2>t_1+N$ such that $\sigma_{t_2}<1+\sigma$. Next, we show  $\sigma_{t}<1+\sigma$ for all $t\geq t_2$ by  induction on set $\lbrace |F|<+\infty \rbrace\cap D$.

Suppose that $\sigma_k<1+\sigma$ for some $k\geq t_2$, then  \dref{tn} gives
\begin{eqnarray}
\sigma_{k+1}^2&=&\sigma^2\cdot\frac{|P_{k+2}^{-1}|}{|P_{k+1}^{-1}|}=\sigma^2+\phi_{k+1}^{\tau}P_{k+1}\phi_{k+1}\leq \sigma^2+\frac{|\phi(y_{k+1})|^2}{\lambda_{\min}(k+1)}\nonumber\\
&\leq & \sigma^2+\frac{n(K_1+K_2)+nK_{2}|y_{k+1}|^{2b_1}}{\varepsilon k}\nonumber\\
&\leq &\sigma^2+\frac{n(K_1+K_2)+nK_{2}\log^{2b} k \cdot\sigma_{k}^{2b}}{\varepsilon k}\nonumber\\
&<&(1+\sigma)^2,\quad \mbox{a.s.}\quad\mbox{on}~\lbrace |F|<+\infty \rbrace\cap D.\nonumber
\end{eqnarray}
By induction, we know $\sigma_{t}<1+\sigma$ for all $t\geq t_2$. This means
\begin{eqnarray*}
\sup_{t}\sigma_t<+\infty\quad \mbox{a.s.}\quad\mbox{on}~\lbrace |F|<+\infty \rbrace\cap D.
\end{eqnarray*}

So the remainder of the argument is focused on  set   $\lbrace |F|=+\infty \rbrace\cap D$.  By Lemma \ref{tzz1} again, there exists a random integer $t_1'$ such that
\begin{eqnarray}\label{ltt}
\lambda_{\min}(t+1)&\geq & M\sum_{i=1}^{t}\frac{1}{\sigma_{i-1}}-\varepsilon t,\quad t>t_1',\quad \mbox{a.s.}.
\end{eqnarray}
On $\lbrace |F|=+\infty \rbrace\cap D$, select an random integer $k'\in F $ such that $k'>t_1'+N+2$. we now  prove that for all $t\geq k'$,
\begin{eqnarray}\label{tj1}
\lambda_{\min}(t+1)>\varepsilon t\quad\mbox{and}\quad \sigma_{t}<1+\sigma\quad \mbox{a.s.}\quad\mbox{on}~\lbrace |F|=+\infty \rbrace\cap D.
\end{eqnarray}
As a matter of fact, for $k'\in F $, we have
\begin{equation*}
\sigma_{k'}^2=\sigma^2\cdot\frac{|P_{k'+1}^{-1}|}{|P_{k'}^{-1}|}<\sigma^2\cdot(1+\sigma^{-2})^{k'+1- k'}<(1+\sigma)^2.
\end{equation*}
Consequently,  \dref{ltt} yields
\begin{eqnarray}
\lambda_{\min}(k'+1)&\geq & M\sum_{i=1}^{k'}\frac{1}{\sigma_{i-1}}-\varepsilon k'\geq M\cdot\sigma\cdot k'\cdot\left(\frac{1}{|P_{k'}^{-1}|}\right)^{\frac{1}{2k'}}-\varepsilon k'\nonumber\\
&\geq &  M\cdot\sigma\cdot k'\cdot\left(\frac{1}{|P_{k'+1}^{-1}|}\right)^{\frac{1}{2k'}}-\varepsilon k'\nonumber\\
&> & M\cdot\sigma\cdot k'\cdot\left(\frac{1}{(1+\sigma^{-2})^{k'+1}}\right)^{\frac{1}{2k'}}-\varepsilon k'\geq \varepsilon k'.\nonumber
\end{eqnarray}

Assume that for some $j\geq k'$, \dref{tj1} holds for all $t\in [k', j]$. Then it follows that
\begin{eqnarray}
\sigma_{j+1}^{2}&=&\sigma^2\cdot\frac{|P_{j+2}^{-1}|}{|P_{j+1}^{-1}|}\leq \sigma^2+\frac{|\phi(y_{j+1})|^2}{\lambda_{\min}(j+1)}\leq  \sigma^2+\frac{m_1(K_1+K_2)+m_1K_{2}|y_{j+1}|^{2b}}{\varepsilon j}\nonumber\\
&\leq &\sigma^2+\frac{n(K_1+K_2)+nK_{2}\log^{2b} j \cdot\sigma_{j}^{2b}}{\varepsilon j}
<(1+\sigma)^2\quad \mbox{a.s.}\quad\mbox{on}~\lbrace |F|=+\infty \rbrace\cap D.\nonumber
\end{eqnarray}
As a result,
\begin{eqnarray}
\lambda_{\min}(j+2)&\geq & M\sum_{i=1}^{j+1}\frac{1}{\sigma_{i-1}}-\varepsilon(j+1)\nonumber\\
&=&M\sum_{i=1}^{k'}\frac{1}{\sigma_{i-1}}+M\sum_{k'<i\leq j+1}\frac{1}{\sigma_{i-1}}-\varepsilon(j+1)\nonumber\\
&> & M\cdot\sigma\cdot k'\cdot\left(\frac{1}{(1+\sigma^{-2})^{k'+1}}\right)^{\frac{1}{2k'}}+\frac{M}{1+\sigma}(j-k+1)-\varepsilon(j+1)\nonumber\\
&\geq &\varepsilon(j+1)\quad \mbox{a.s.}\quad\mbox{on}~\lbrace |F|=+\infty \rbrace\cap D.\nonumber
\end{eqnarray}
Therefore, \dref{tj1} is true for $t=j+1$, and the induction is completed. So
\begin{eqnarray*}
\sup_{t}\sigma_t<+\infty\quad \mbox{a.s.}\quad\mbox{on}~\lbrace |F|=+\infty \rbrace\cap D.
\end{eqnarray*}

To sum up,  \dref{jj} holds as desired.
\end{proof}


\begin{lemma}\label{zh}
Under Assumptions A1--A2, if \dref{fk1k2} holds and $\sup_{t}\sigma_t<+\infty$ a.s., then
\begin{eqnarray*}
\sup_{t\geq 1}\frac{1}{t}\sum_{i=1}^{t}y_i^2<+\infty,\quad \mbox{a.s.}.
\end{eqnarray*}
\end{lemma}
\begin{proof}
Recall from \cite[Lemma 3.1]{guo95}  that
\begin{equation}\label{alpha}
\sum_{i=0}^{t}\alpha_i=O\left(\log\left(1+\sum_{i=0}^{t}\phi_{i}^{\tau}\phi_{i}\right)\right),\quad \mbox{a.s.},\nonumber
\end{equation}
where $\alpha_{i}\triangleq{(1+\phi_i^{\tau}P_i\phi_i)^{-1}(\tilde{\theta}\phi_i)^2}$, $i\geq 0 $.  Therefore,
\begin{eqnarray*}
\frac{1}{2}\sum_{i=1}^{t}y_i^2-\sum_{i=0}^{t}w_{i+1}^2&\leq &\sum_{i=0}^{t}(y_{i+1}-w_{i+1})^2=\sum_{i=0}^{t}\alpha_i\frac{|P_{i+1}^{-1}|}{|P_{i}^{-1}|}=O\left(\sum_{i=0}^{t}\alpha_i\right)\nonumber\\
&=&O\left(\log\left(1+\sum_{i=0}^{t}\phi_{i}^{\tau}\phi_{i}\right)\right)\nonumber\\
&\leq &O\left(\log\left(1+k_1^2\sum_{i=0}^{t}e^{2k_2|y_i|}\right)\right)\nonumber\\
&\leq & O\left(\log\left(1+k_1^2e^{2k_2|y_0|}+k_1^2t\cdot e^{2k_2\cdot\max_{1\leq i\leq t}|y_i|}\right)\right)\nonumber\\
&=&O(1)+O(\log t)+O\left(\left(\sum_{i=1}^{t}y_i^2\right)^{\frac{1}{2}}\right),\quad \mbox{a.s.}.
\end{eqnarray*}
Observe that $\sum_{i=0}^{t}w_{i+1}^2=O(t)$ as $t\rightarrow+\infty$, then
\begin{eqnarray*}
\frac{1}{2}\sum_{i=1}^{t}y_i^2\leq O(t)+O\left(\left(\sum_{i=1}^{t}y_i^2\right)^{\frac{1}{2}}\right),
\end{eqnarray*}
which implies $\sum_{i=1}^{t}y_i^2=O(t)$ almost surely.
\end{proof}
With all the technique  lemmas ready, Theorem \ref{1sta} is straightforward.

\begin{proof}[Proof of Theorem \ref{1sta}]
Taking account of Lemmas \ref{one} and \ref{two}, we have
\begin{equation}
\sup\limits_{t} \sigma_{t}<+\infty,\quad \mbox{a.s.},\nonumber
\end{equation}
which leads to Theorem \ref{1sta}  directly by Lemma \ref{zh}.
\end{proof}

\section{Proof of Theorem \ref{2sta}}\label{AppA1}
The proof is based on  two lemmas below.
\begin{lemma}\label{one1}
If $\liminf_{t\to+\infty}\frac{\lambda_{\min}(t+1)}{t}>0$ and $\sup\limits_{t}\sigma_t=+\infty$ hold almost surely on a set $D$ with $P(D)>0$, then
\begin{eqnarray*}
\lim_{t\rightarrow +\infty}\sigma_t=+\infty,\quad \mbox{a.s.}\quad \mbox{on}~D.
\end{eqnarray*}
\end{lemma}
\begin{proof} Given a number $z>\sigma^2$,  define
\begin{equation}
\Omega_{k+1}\triangleq{\left\lbrace\sigma_{k}^2\leqslant z,\sigma_{k+1}^2\geqslant z\right\rbrace},\quad k\geq 0.\nonumber
\end{equation}
Therefore,
\begin{eqnarray}\label{qm2}
&&P(\Omega_{k+1}|\mathcal{F}_{k}^y)\nonumber\\
&=&P\left(\sigma^2\cdot\frac{|P_{k+2}^{-1}|}{|P_{k+1}^{-1}|}\geqslant z,\sigma_{k}^2\leqslant z\Big|\mathcal{F}_{k}^y\right)\nonumber\\
&\leq &P\left(\sigma^2+\frac{\|\phi(y_{k+1})\|^2}{\lambda_{\min}(k+1)}\geqslant z,\sigma_{k}^2\leqslant z\bigg|\mathcal{F}_{k}^y\right)\nonumber\\
&\leq &P\left(|y_{k+1}|\geqslant \frac{1}{k_2}\log\left(\frac{(z-\sigma^2)\lambda_{\min}(k+1)}{k_1}\right),\sigma_{k}^2\leqslant z\bigg|\mathcal{F}_{k}^y\right)\nonumber\\
&=&E\left\lbrace I_{\left\lbrace |y_{k+1}|\geqslant  \frac{1}{k_2}\log\left(\frac{(z-\sigma^2)\lambda_{\min}(k+1)}{k_1}\right)\right\rbrace}\cdot I_{\left\lbrace \sigma_{k}^2\leqslant z\right\rbrace}\bigg|\mathcal{F}_{k}^y\right\rbrace\nonumber\\
&=&I_{\left\lbrace \sigma_{k}^2\leqslant z\right\rbrace}\cdot E\left\lbrace I_{\left\lbrace |y_{k+1}|\geqslant  \frac{1}{k_2}\log\left(\frac{(z-\sigma^2)\lambda_{\min}(k+1)}{k_1}\right)\right\rbrace}\bigg|\mathcal{F}_{k}^y\right\rbrace\nonumber\\
&=&I_{\left\lbrace \sigma_{k}^2\leqslant z\right\rbrace}\cdot \frac{1}{\sqrt{2\pi}}\int_{|x\cdot\sigma_{k}|\geqslant  \frac{1}{k_2}\log\left(\frac{(z-\sigma^2)\lambda_{\min}(k+1)}{k_1}\right)}e^{-\frac{x^2}{2}}\,dx\nonumber\\
&\leq &\frac{1}{\sqrt{2\pi}}\int_{|x\cdot\sqrt{z}|\geqslant  \frac{1}{k_2}\log\left(\frac{(z-\sigma^2)\lambda_{\min}(k+1)}{k_1}\right)}e^{-\frac{x^2}{2}}\,dx\nonumber\\
&=&\frac{1}{\sqrt{2\pi}}\int_{|x|\geq X_{k}}e^{-\frac{x^2}{2}}\,dx,
\end{eqnarray}
where $X_{k}\triangleq{\frac{1}{\sqrt{z}} \frac{1}{k_2}\log\left(\frac{(z-\sigma^2)\lambda_{\min}(k+1)}{k_1}\right)}$. Since $\liminf_{k\to+\infty}\frac{\lambda_{\min}(k+1)}{k}>0$ implies $\liminf_{k\to+\infty}\frac{X_{k}}{\log k}>0$, \dref{qm2} yields
\begin{equation}
\sum_{k=1}^{+\infty}P(\Omega_{k+1}|\mathcal{F}_{k}^y)\leq \sum_{k=1}^{+\infty}\frac{1}{\sqrt{2\pi}}\int_{|x|\geq X_k}e^{-\frac{x^2}{2}}\,dx<+\infty.\nonumber
\end{equation}
Taking account of \textit{Borel-Cantelli-Levy} theorem, $\lbrace \Omega_{k}\rbrace$ occur only finite times  almost surely.
The rest of the proof is  as the same of that for \cite[Lemma 3.5]{liu2018}.
\end{proof}
\begin{lemma}\label{bl}
Let $\lbrace A_k\rbrace_{k\geq1}$ be a sequence of events that $A_k\triangleq{\lbrace y_{k}\in S_{L}\rbrace}$.
Then, there exists a constant $c>0$, which only depends on $f_{1},\ldots,f_{n}$, such that for all sufficiently large $t$,
\begin{equation}\label{zbl}
\sum_{k=1}^{t}I_{A_k}\geq ct,\quad \mbox{a.s.}.
\end{equation}
\end{lemma}

\begin{proof}
Recall that \dref{sl} means there are two numbers $ q_1,q_2>0$ such that 
\begin{equation}\label{slq}
\frac{\ell(S_{L}\cap [-l,l])}{l}> q_2\quad\mbox{for}\quad l\geq q_1.
\end{equation}
 Since $y_{i+1}$ is conditional Gaussian with the conditional mean $m_i=0$ and variance $\sigma_i^2$ from \dref{mean} and \dref{var}, we compute
\begin{eqnarray}
P\left(A_{i+1}|\mathcal{F}_i^y\right)&=&\frac{1}{\sqrt{2\pi}}\int_{|x\sigma_i|\in S_{L}}e^{-\frac{x^2}{2}}\,dx\nonumber\\
&\geq &\frac{1}{\sqrt{2\pi}}\int_{|x\sigma_i|\in S_{L},|x|\leqslant \frac{q_1}{\sigma}}e^{-\frac{x^2}{2}}\,dx\nonumber\\
&\geq &\ell\left(x:|x\sigma_i|\in S_{L},|x|\leqslant \frac{q_1}{\sigma}\right)\cdot \frac{1}{\sqrt{2\pi}}e^{-\frac{q_1^2}{2\sigma^2}}\nonumber\\
&=&\dfrac{\ell(S_{L}\cap [-q_1\sigma_i\sigma^{-1},q_1\sigma_i\sigma^{-1}])}{\sigma_i}\cdot \frac{1}{\sqrt{2\pi}}e^{-\frac{q_1^2}{2\sigma^2}}.\nonumber
\end{eqnarray}
Owing to \dref{slq} and $\sigma_i\geq \sigma$, we immediately deduce $P(A_{i+1}|\mathcal{F}_i^y)>\frac{q_1q_2}{\sqrt{2\pi}\sigma}e^{-\frac{q_1^2}{2\sigma^2}}$. Further, by applying the strong
law of large numbers for the martingale differences, we have
\begin{equation*}
\frac{\sum_{k=1}^{t}(I_{A_k}-P(A_{k}|\mathcal{F}_{k-1}^y))}{t}=o(1),\quad \mbox{a.s.}.
\end{equation*}
Then for all sufficiently large $t$,
\begin{equation*}
\sum_{k=1}^{t}I_{A_k}\geq \sum_{k=1}^{t}P(A_{k}|\mathcal{F}_{k-1}^y)-\frac{q_1q_2}{2\sqrt{2\pi}\sigma}e^{-\frac{q_1^2}{2\sigma^2}}t\geq \frac{q_1q_2}{2\sqrt{2\pi}\sigma}e^{-\frac{q_1^2}{2\sigma^2}}t.
\end{equation*}
So Lemma \ref{bl}  follows by letting $c=\frac{q_1q_2}{2\sqrt{2\pi}\sigma}e^{-\frac{q_1^2}{2\sigma^2}}$.
\end{proof}

\begin{proof}[Proof of Theorem \ref{2sta}]
Same as previous, we are going to show
 \begin{equation}\label{xyzw}
\sup_{t}\sigma_t<+\infty,\quad \mbox{a.s.}.
\end{equation}
Assume  there is a set $D$ with $P(D)>0$ such that $\sup_{t}\sigma_t=+\infty$ on $D$. Note that $y_k\in S_L$ infers
\begin{equation*}
\sigma_{k}^{2}=\frac{|P_{k+1}^{-1}|}{|P_{k}^{-1}|}\leq \sigma^2+\frac{\|\phi(y_k)\|^2}{\lambda_{\min}(k)}\leq \sigma^2+L^2,
\end{equation*}
then for any $t\geq 1$,
\begin{equation}\label{ia}
\sum_{k=1}^{t}I_{\lbrace \sigma_{k}\leq\sqrt{\sigma^2+L^2}\rbrace}\geq\sum_{k=1}^{t}I_{A_k},
\end{equation}
where $I_{A_k}, k\in[1,t]$ are defined in Lemma \ref{bl}.
Take $\varepsilon=\frac{1}{4}(\sigma^2+L^2)^{-\frac{1}{2}}Mc$.
By Proposition \ref{tzz1},  Lemma \ref{bl} and \dref{ia}, for all sufficiently large $t$, we have
\begin{eqnarray}
\lambda_{\min}(t+2)&\geq & M\sum_{i=1}^{t+1}\frac{1}{\sigma_{i-1}}-\varepsilon (t+1)\nonumber\\
&\geq & M\cdot\frac{ct}{\sqrt{\sigma^2+L^2}}-\frac{1}{4}(\sigma^2+L^2)^{-\frac{1}{2}}Mc(t+1)\nonumber\\
&>& \frac{Mct}{2\sqrt{\sigma^2+L^2}},  \quad \mbox{a.s.}.\nonumber
\end{eqnarray}
This means  $\liminf_{t\rightarrow+\infty}\frac{\lambda_{\min}(t+1)}{t}>0$ almost surely. So,  Lemma \ref{one1}  gives
$$\lim_{t\rightarrow +\infty}\sigma_t=+\infty,\quad \mbox{a.s.}\quad \mbox{on}~D.$$
According to \dref{ia}, it  turns out that
\begin{eqnarray*}
\limsup_{t\rightarrow+\infty}\frac{\sum_{k=1}^{t}I_{A_k}}{t}\leq\limsup_{t\rightarrow+\infty}\frac{\sum_{k=1}^{t}I_{\lbrace \sigma_{k}\leq\sqrt{\sigma^2+L^2}\rbrace}}{t}=0,\quad \mbox{a.s.}\quad \mbox{on}~D,
\end{eqnarray*}
which contradicts to \dref{zbl}. We thus conclude  \dref{xyzw} and Lemma \ref{zh} applies.
\end{proof}

%





\appendix
\section{Proof of Lemma \ref{c1} and Proposition \ref{tzz1}}\label{AppA}
To show Lemma \ref{c1}, we need a simple fact.
\begin{lemma}\label{dj}
Assume the conditions of Lemma \ref{c1} hold,
then
\begin{eqnarray*}
\sup_{t\geq 1}\frac{1}{t}\sum_{i=1}^{t}y_i^2<+\infty \quad \mbox{is equaivalent to}\quad   \sup_{t}\sigma_t<+\infty \quad \mbox{a.s.}.
\end{eqnarray*}
\end{lemma}

\begin{proof}
Denote
\begin{eqnarray*}
W\triangleq{\left\lbrace\sup_{t\geq 1}\frac{1}{t}\sum_{i=1}^{t}y_i^2<+\infty\right\rbrace}.
\end{eqnarray*}
By  \textit{Hadamard} inequality, for all sufficiently large $t$,
\begin{eqnarray}\label{po}
|P_{t}^{-1}|&\leq &\prod_{j=1}^{n}\left(1+\frac{1}{\sigma^2}\sum_{i=0}^{t-1}f_{1j}^2(y_i)\right)=\prod_{j=1}^{n}\left(1+O(t)+O\left(\sum_{i=0}^{t-1}y_i^{2b}\right)\right)\nonumber\\
&\leq & \prod_{j=1}^{n}\left(1+O(t)+O\left(\left(\sum_{i=0}^{t-1}y_i^{2}\right)^b\right)\right)=O\left(t^{(b+1)n}\right),\quad\mbox{on}~W.
\end{eqnarray}
So,   applying Lemma \ref{two}  infers
\begin{eqnarray}
\sup_{t}\sigma_t<+\infty\quad \mbox{a.s.}\quad\mbox{on}~W.\nonumber
\end{eqnarray}
Since the converse part is verified by Lemma \ref{zh}, the lemma is proved.
\end{proof}

\begin{proof}[Proof of Lemma \ref{c1}]
Let $W$ be defined in Lemma \ref{dj} and $\varepsilon=\frac{M}{2\sigma}$. By Proposition \ref{tzz1},
\begin{eqnarray}\label{pmt}
\lambda_{\min}(t+1)&\geq & M\sum_{i=1}^{t}\frac{1}{\sigma_{i-1}}-\frac{M}{2\sigma}t\nonumber\\
&\geq & M\sigma^{-1}t\left(\frac{1}{|P_{t}^{-1}|}\right)^{\frac{1}{2t}}-\frac{M}{2\sigma}t,\quad \mbox{a.s.}\quad\mbox{on}~W,
\end{eqnarray}
where $t$ is sufficiently large.
Combining  \dref{po} and \dref{pmt}, we obtain 
\begin{eqnarray}\label{diyi}
\liminf_{t\rightarrow+\infty}\frac{\lambda_{\min}(t+1)}{t}>0,\quad \mbox{a.s.}\quad\mbox{on}~W.
\end{eqnarray}
Define a martingale difference sequence
\begin{equation*}
Z_{i}=\frac{1}{i}(y_{i}^{2}-E(y_{i}^2|\mathcal{F}_{i-1}^{y}))=\frac{1}{i}(y_{i}^{2}-\sigma_{i-1}^2),\quad i\geq 1.
\end{equation*}
In view of Lemma \ref{dj},
\begin{eqnarray*}
\sum_{i=1}^{+\infty}E(Z_{i}^2|\mathcal{F}_{i-1}^{y})=\sum_{i=1}^{+\infty}\frac{2\sigma_{i-1}^4}{i^2}<+\infty,\quad\mbox{a.s.}\quad\mbox{on}~W.
\end{eqnarray*}
By \cite[Theorem 2.7]{cg91}, we deduce that $\sum_{i=1}^{+\infty}Z_i$ converges almost surely on $W$. This, together with  \textit{Kronecker Lemma}, leads to
\begin{equation}\label{yo}
\sum_{i=1}^{t}(y_{i}^{2}-\sigma_{i-1}^2)=o(t),\quad\mbox{a.s.}\quad\mbox{on}~W.
\end{equation}

By \dref{bx}, \dref{ylog}, \dref{diyi} and Lemma \ref{dj}, as long as $t$ is sufficiently large,
\begin{eqnarray}
\sigma_{t+1}^2&\leq &\sigma^2+\frac{\|\phi(y_{t+1})\|^2}{\lambda_{\min}(t+1)}=\sigma^2+O\left(\frac{y_{t+1}^{2b}}{t}\right)\nonumber\\
&=&\sigma^2+O\left(\frac{(\sigma_t\log(t+1))^{2b}}{t}\right)=\sigma^2+o(1),\quad \mbox{a.s.}\quad\mbox{on}~W.\nonumber
\end{eqnarray}
Recall that $ \sigma_{t}\geq \sigma$,  therefore
\begin{equation}\label{1o}
\lim_{t\rightarrow\infty}\sigma_{t}=\sigma,\quad \mbox{a.s.}\quad\mbox{on}~W.
\end{equation}
According to \dref{yo} and \dref{1o}, we derive
\begin{eqnarray*}
\lim_{t\rightarrow\infty}\frac{1}{t}\sum_{i=1}^{t}y_{i}^{2}=\lim_{t\rightarrow\infty}\frac{1}{t}\sum_{i=1}^{t}\sigma_{i-1}^2=\sigma^2,\quad \mbox{a.s.}\quad\mbox{on}~W,
\end{eqnarray*}
which together with \dref{diyi} gives Lemma \ref{c1}.
\end{proof}

\begin{proof}[Proof of Proposition \ref{tzz1}]
The proof is a minor modification of that of \cite{liuli2020}. So, we  shall follow the notations of \cite{liuli2020} and  only write the differences.  Let $\delta$,  $U_x$ and $S_j(q)$  be defined in   \cite[Section 3.1]{liuli2020}. Then,
\begin{equation*}
\inf_{\|x\|=1}\ell\left(\lbrace y: |\phi^{\tau}(y)x|> \delta\rbrace\cap \bigcup_{j=1}^{p}S_j(q)\right)>0.
\end{equation*}
For any $x\in\mathbb{R}^{\alpha}$ with $\|x\|=1$,
\begin{eqnarray}\label{wxf}
P(y_{i}\in U_x|\mathcal{F}_{i-1}^{y})&= &\frac{1}{\sqrt{2\pi}}\int_{s\sigma_{i-1}\in U_x}e^{-\frac{s^2}{2}}\,ds\nonumber\\
&= & \frac{1}{\sqrt{2\pi}}\int_{s\sigma_{i-1}\in U_x,|s|\leq \sigma^{-1}R}e^{-\frac{s^2}{2}}\,ds\nonumber\\
&\geq &\frac{1}{\sqrt{2\pi}}\cdot\frac{\ell(U_x)}{\sigma_{i-1}}\cdot e^{-\frac{(\sigma^{-1}R)^2}{2}}\nonumber\\
&\geq &\frac{1}{\sigma_{i-1}}\cdot\frac{e^{-\frac{(\sigma^{-1}R)^2}{2}}}{\sqrt{2\pi}}\cdot \inf_{\|x\|=1}\ell\left(\lbrace y: |\phi^{\tau}(y)x|> \delta\rbrace\cap \bigcup_{j=1}^{p}S_j(q)\right)\nonumber\\
&\triangleq & \frac{k_1}{\sigma_{i-1}},
\end{eqnarray}
where
$R=\mbox{dist}(\bigcup_{j=1}^{p}S_j(q)).$

Now, we modify Section 3.3 of \cite{liuli2020} to deduce out result.
To this end,
 for any $\epsilon>0$, recalling the definition of $g_{\epsilon}$ in  \cite[ Lemma 3.12]{liuli2020}, the strong large number laws for martingale differences shows that  all $g_{\epsilon}\in\mathcal{G}_{\epsilon}$ fulfill
\begin{eqnarray*}
\lim_{t\rightarrow\infty}\frac{1}{t}\sum_{i=1}^{t}g_{\epsilon}(i)+\epsilon=0,\quad \mbox{a.s.}.
\end{eqnarray*}
Since there are only  finite $U_{\epsilon}$ satisfying $U_{\epsilon}\subset\bigcup_{j=1}^{p}S_{j}(q)$, we conlude
\begin{eqnarray}
\lim_{t\rightarrow\infty}\inf_{U_{\epsilon}\subset\cup_{j=1}^{p}S_{j}(q)}\frac{1}{t}\sum_{i=1}^{t}g_{\epsilon}(i)=-\epsilon,\quad \mbox{a.s.},\nonumber
\end{eqnarray}
which, together with \cite[Lemma 3.12(ii)]{liuli2020}, yields
\begin{eqnarray*}
\liminf_{t\rightarrow+\infty}\inf_{\|x\|=1}\frac{1}{t}\sum_{i=1}^{t}g_{x}(i)
&\geq &\liminf_{t\rightarrow+\infty}\inf_{\|x\|=1}\frac{1}{t}\sum_{i=1}^{t}g_{\epsilon}^{x}(i)\nonumber\\
&\geq & \liminf_{t\rightarrow\infty}\inf_{U_{\epsilon}\subset\bigcup_{j=1}^{p}S_{j}(q)}\frac{1}{t}\sum_{i=1}^{t}g_{\epsilon}(i)\nonumber\\
&=&-\epsilon,\quad \mbox{a.s.}.
\end{eqnarray*}
Furthermore, by the arbitrariness of $\epsilon$, we obtain
\begin{equation}\label{infa}
\liminf_{t\rightarrow+\infty}\inf_{\|x\|=1}\frac{1}{t}\sum_{i=1}^{t}g_{x}(i)\geq 0\quad \mbox{a.s.}.
\end{equation}
Combining \eqref{infa} and \eqref{wxf}, for any given $\varepsilon>0$, there exists a random integer $T>0$ such that for all $t>T$,
\begin{eqnarray*}
\frac{1}{t}\sum_{i=1}^{t}I_{\lbrace y_i\in U_x\rbrace}&>&\frac{1}{t}\sum_{i=1}^{t}P(y_{i}\in U_x|\mathcal{F}_{i-1}^{y})-\frac{\sigma^2}{\delta^2}\cdot\varepsilon\nonumber\\
&\geq & \frac{1}{t}\sum_{i=1}^{t}\frac{k_1}{\sigma_{i-1}}-\frac{\sigma^2}{\delta^2}\cdot\varepsilon.
\end{eqnarray*}
Then
\begin{eqnarray*}
\lambda_{\min}(t+1)&=&\inf_{\|x\|=1}x^{\tau}\left(I_n+\frac{1}{\sigma^2}\sum_{i=0}^{t}\phi_{i}\phi_{i}^{\tau}\right)x\geq\frac{1}{\sigma^2}\sum_{i=1}^{t}(\phi^{\tau}(y_i)x)^2\nonumber\\
&\geq &\frac{\delta^2}{\sigma^2}\cdot\left(\sum_{i=1}^{t}\frac{k_1}{\sigma_{i-1}}-\frac{\sigma^2}{\delta^2}\cdot\varepsilon t\right)=\frac{\delta^2}{\sigma^2}\cdot k_1\sum_{i=1}^{t}\frac{1}{\sigma_{i-1}}-\varepsilon t.
\end{eqnarray*}
Let $M=\frac{\delta^2}{\sigma^2}\cdot k_1$ and Proposition \ref{tzz1} follows.
\end{proof}


%
\bibliographystyle{siamplain}

\end{document}